\documentclass{amsart}

\newtheorem{theorem}{Theorem}[section]
\newtheorem{lemma}[theorem]{Lemma}

\theoremstyle{definition}

\newtheorem{corollary}[theorem]{Corollary}
\newtheorem{proposition}[theorem]{Proposition}
\newcommand{\R}{{\Bbb R}}

\theoremstyle{remark}
\newtheorem{remark}[theorem]{Remark}

\numberwithin{equation}{section}

\begin{document}

\title[3/2-stability theorem  for the KPP-Fisher delayed equation]{An extension of the Wright's 3/2-theorem for the KPP-Fisher delayed equation}
\author{Karel Hasik}
\address{Mathematical Institute, Silesian University, 746 01 Opava, Czech Republic}
\email{Karel.Hasik@math.slu.cz}
\thanks{This research was realized within the framework of the OPVK program, project CZ.1.07/2.300/20.0002. The second author was also partially supported  by FONDECYT (Chile), project 1110309, and by
CONICYT (Chile) through PBCT program ACT-56.}
\author{Sergei Trofimchuk}
\address{Instituto de Matem\'atica y Fisica, Universidad de Talca, Casilla 747,
Talca, Chile}
\email{trofimch@inst-mat.utalca.cl}

\begin{abstract}We present a short proof of the following natural extension of the famous Wright's $3/2$-stability theorem: the conditions $\tau \leq 3/2, \ c \geq 2$ imply the  presence of the positive traveling fronts (not necessarily monotone) $u = \phi(x\cdot \nu+ct), \ |\nu| =1,$ in the
delayed KPP-Fisher equation
$u_t(t,x) = \Delta u(t,x)  + u(t,x)(1-u(t-\tau,x)), $ $u \geq 0,$ $x
\in \R^m.$   
\end{abstract}

\maketitle

\section{Introduction and main result}\label{intro} 
\noindent The delayed KPP-Fisher (i.e. Kolmogorov-Petrovskii-Piskunov-Fisher)  equation  \begin{equation}\label{17}
u_t(t,x) = \Delta u(t,x)  + u(t,x)(1-u(t-\tau,x)), \ u \geq 0,\ x \in
\R^m, \tau\geq 0, \end{equation} 
is one of the most conspicuous
examples of delayed reaction-diffusion equations. During the past decade, this model  together with the following non-local version of the KPP-Fisher equation 
\begin{equation}\label{17nl} 
u_t(t,x) = \Delta u(t,x)  + u(t,x)\left(1-\int_{\R^n} K(y) u(t,x-y)dy\right), \ \int_{\R^n} K(s)ds=1. 
\end{equation}
have been intensively  studied by many authors,  see e.g.  
\cite{BNPR,FW,FZ,fhw,GT,HT,KO,wz}.    

One of the key topics related to equations (\ref{17}), (\ref{17nl}) concerns  the existence and further properties of smooth positive  traveling front solutions  $u(x,t) = \phi(\nu \cdot  x +ct),$ $|\nu| =1$ for (\ref{17}).  It is supposed that $c>0$ and that  the profile
$\phi$ satisfies the boundary conditions $\phi(-\infty) = 0$, $\phi(+\infty) = 1$. 
A few years ago,  not much was known about the conditions guaranteeing the existence 
of these wavefronts in (\ref{17}).  Several   existence results having rather partial character were provided  in \cite{wz} (for each $c>2$ and $\tau \in [0,\tau(c)]$ with sufficiently small $\tau(c)$) and in \cite{fhw,FTnl} (for each $\tau  \leq 3/2$ and  $c \geq c(\tau)$ with sufficiently large $c(\tau)$).  In this respect, a significant progress 
was achieved only very recently when the existence and uniqueness problems for (\ref{17}), (\ref{17nl}) were completely solved for the case of monotone profiles \cite{BNPR,FW,FZ,GT,HT,KO}.  However, the monotonicity of $\phi$  is  a rather restrictive assumption:  it is clear that  traveling fronts of (\ref{17}), (\ref{17nl})  that  oscillate around $1$ at $+\infty$ (hence, non-monotone ones)   comprise the largest part of  the set of all wavefront solutions \cite{ZAMP,BNPR,HT,NPT}.  
In this note, by establishing an `almost optimal criterion'  for the presence of oscillating fronts in equation (\ref{17}), we achieve an essential improvement of the existence results from \cite{FTnl,GT,KO,wz}.   Still, the complete solution of the mentioned problem remains to be a quite challenging project which is directly  connected  to  the long standing Wright's global stability conjecture \cite{TK,EMW}.  

Let us explain the last comment in  more detail. Indeed,  looking for a wave solution of  (\ref{17}) in the slightly modified form $u(t,x) = \psi(\sqrt{\epsilon}x+t), \ \epsilon = 1/c^2,$ $ \psi (s) = \phi(cs)$, we find that 
\begin{equation}\label{eWE}
\epsilon \psi''(t) - \psi'(t) +\psi(t)(1-\psi(t-\tau)) =0, \ t \in \R.  
\end{equation}
In the limit case $\epsilon =0$ equation (\ref{eWE}) is called the Hutchinson's equation and it was 
conjectured by E.M. Wright  \cite{EMW} that the steady state $\psi =1$ of
 (\ref{eWE}) with ${\epsilon =0}$ is globally stable in the domain of all positive solutions  $\psi >0$ if and only if  $\tau \leq \pi/2$.  A weaker version of the Wright's conjecture can be also considered:  the Hutchinson's equation has a positive heteroclinic connection (i.e. traveling front type solution) if and only if  $\tau \leq \pi/2$.  The both conjectures are supported by the 
 {\it `very difficult theorem of Wright'} (the quoted phrase is from the Jack  Hale's book \cite[p.64]{JH}) proved in  \cite{EMW}:  If $\tau \leq 3/2$ then the positive equilibrium of  (\ref{eWE}) with ${\epsilon =0}$  is globally stable in the domain of positive solutions.  Remarkably, as it was shown in \cite{fhw,FTnl} by  means of 
the Hale-Lin approach,  the Wright's $3/2$-theorem  can be extended to (\ref{eWE}) with $\epsilon >0$ in the following way:  equation   (\ref{eWE}) has a positive heteroclinic connection for each positive fixed $\tau \leq 3/2$ if  $\epsilon>0$ is sufficiently small.  The main result of this work shows that the smallness condition on $\epsilon$ (i.e. the requirement that the propagation speed $c$ has to be sufficiently large) can be avoided and that the full analog of the Wright's theorem holds for  (\ref{17}):
\begin{theorem} \label{eWT}
Assume that $c \geq 2$ and $\tau \in [0, 3/2]$. Then the delayed KPP-Fisher  reaction-diffusion equation  (\ref{17}) has at least one positive  traveling front solution. 
\end{theorem}
It is well known that  the inequality $c\geq 2$ is mandatory for the existence of 
positive wavefronts \cite{BNPR,GT,HT}. We also believe that, similarly to the monotone fronts \cite{FW,FZ,GT,HT},  there is a unique  (up to a translation) oscillating front  for each fixed $c$.  

Theorem \ref{eWT} strongly  supports the next generalisation of the weak Wright's conjecture  \cite{HT}:  equation  (\ref{17}) has at least one positive  traveling front $u = \phi(\nu \cdot  x +ct),$ $|\nu| =1$,   if and only if $c \geq 2$ and the equation $\lambda^2-c\lambda - e^{-\lambda c\tau} =0$ has a unique root $\lambda$ with the positive real part.  In particular, this means  that the maximal possible improvement of the interval $[0,3/2]$ in  Theorem \ref{eWT}  is   $[0,\pi/2]$, see \cite[Figure 1]{HT}. 

The starting point for the proof of Theorem \ref{eWT} is the fact that, for each $c\geq 2$, equation  (\ref{17}) (similarly to equation (\ref{17nl}), see \cite{BNPR})
has at least one positive wave solution  $u = \phi(\nu \cdot  x +ct),$ $|\nu| =1$, satisfying the boundary conditions $\phi(-\infty)=0$, $0 <m= \lim\inf_{t\to+\infty} \phi(t) \leq \lim\sup_{t\to+\infty} \phi(t) =M < +\infty$ (i.e. a semi-wavefront), see \cite{HT}.  The next important fact is that each non-monotone semi-wavefront profile is   sine-like slowly oscillating around $1$ at $+\infty$ \cite{HT,mps,mps2}. 
In Section \ref{BF}, we present several explicit analytic relations limiting  the amplitude of these oscillations. At the first glance, the mentioned restrictions are generated by rather  cumbersome bounding functions. Surprisingly,   
these functions have nice internal structures (previously analysed in  \cite{lpprt}) that allow for their satisfactory description  in Section \ref{AF}.  At the very end of Section \ref{BF},  in order to demonstrate Theorem \ref{eWT},  we show that $\tau \leq 3/2$ together with  $c \geq 2$ imply  $m=M=1$. 

\vspace{1cm}
\section{Auxiliary functions} \label{AF}
Our approach to the proof of  Theorem \ref{eWT} requires the construction of several suitable bounding functions.  These functions are necessary to relate the values of $m$ and $M$ (defined a few lines above); it is clear  that their choice is  by no means unique.  Below, we  present our auxiliary functions and prove their properties  which are later used  in the proof of Theorem \ref{eWT}. 
First, we consider  
$$
\rho(x)= \rho(x,\tau,c):= \tau c f(w(x)), \ \mbox{where} \ f(x): = \frac{-c+\sqrt{c^2+4x}}{2}, \quad  w(x) :=e^{-x}-1. 
$$
\begin{proposition}{\cite{HT}} \label{nsw} \ Let $c \geq 2$. Then the  real analytic  function 
$\rho(x),$ $ x \in \R, $ $ \rho(0) =0, \ \rho(-\infty) = +\infty, \ \rho(+\infty) = -0.5\tau c (-c+\sqrt{c^2-4})<0, $ is strictly decreasing (in fact, $\rho'(x)<0, \ x \in \R$) and has the 
negative Schwarz derivative $(S\rho)(x)$ on $\R$:
$(S\rho)(x)=\rho'''(x)/\rho'(x)-3/2 \left(\rho''(x)/\rho'(x)\right)^2 <0, \quad x \in \R.$
\end{proposition}
It is straightforward to see that $\rho$ is a convex function:
$$
\rho''(x)= c\tau e^{-x}(f'(w(x)))^3(c^2-4+2e^{-x}) >0, \quad x \in \R. 
$$
\begin{corollary} \label{co22} If $c\geq 2$ then,  for all $x >0$, it holds that 
$$
\rho(x) > r(x):= \frac{\rho'(0)x}{1-0.5\rho''(0)x/\rho'(0)} = \frac{-\tau x}{1+0.5(1-2/c^2)x}. 
$$
\end{corollary}
\begin{proof} It is an immediate consequence of Proposition \ref{nsw} and \cite[Lemma 2.1]{lpprt}. 
\end{proof}
Next, for each $c\geq 2, \ \tau \in (1,3/2]$,  we consider
$$
A_-(x,c,\tau) = x+ \rho(x) + \frac{1}{\rho(x)}\int^0_x\rho(s)ds,  \quad x\not=0;
$$
$$
A_+(x,c,\tau) = x+ r(x) + \frac{1}{r(x)}\int^0_xr(s)ds, 
\quad 
 B(x,c,\tau):= 
\frac{1}{r(x)}\int_{-r(x)}^{0}r(s)ds, \quad x > 0. 
$$
It is easy to see that $A_\pm, B$ are continuous at $x=0$ if we set  $A_\pm(0,c,\tau)  = B(0,c,\tau)= 0$.  Observe also that $B(x,c,\tau)$ is strictly decreasing on $\R_+$,  $$A_\pm'(0,c,\tau) = \frac 12 - \tau, \quad A_\pm''(0,c,\tau) = (\tau- \frac 16)(1-\frac{2}{c^2}),$$ 
\begin{eqnarray} \label{AAB}
A_-(x,c,\tau) &< & A_-(x,c,3/2), \   x < 0, \    \tau > 1;  \nonumber \\
A_+(x,c,\tau) &> & A_+(x,c,3/2), \   x > 0,  \  \tau > 1; \\
B(x,c,\tau) &>& B(x,c,3/2), \ x > 0,  \ \tau > 1. \nonumber
\end{eqnarray}
Let $x_2 >0$ be the unique positive solution of equation $-r(x)=x$.  Since $\tau >1$, it holds,  for a positive $x$,  that  $x/r(x) >-1$ if and only if  $x \in (0,  x_2)$.  As it was established in  \cite[Lemma 2.3]{lpprt}, $A_+(x,c,\tau)$ is strictly decreasing in the first variable on $(-\infty,x_2]$.  
The next result has a similar proof: 
\begin{lemma}
\label{31}
$A'_-(x,c, \tau)<0$ and $(SA_-)(x,c,\tau)<0$ once $x/\rho(x) >-1.$
\end{lemma}
\begin{proof} 
Using the convexity of $\rho$ and recalling that $-\rho'(0) =\tau>1, \ \rho(0)=0$, it is easy to see that  $x/\rho(x) >  -1$ if and only if  $x<\bar x_2$ where $\bar x_2$ is the unique positive solution of equation $-\rho(x) = x$. In consequence, $x\rho(x) +\rho^2(x) >0, \ x <\bar x_2, \ x\not=0$, 
$$
A_-'(x,c,\tau)=\rho'(x)\bigg(1-\frac{\int^0_x\rho(s)ds}{\rho^{2}(x)}\bigg) < \rho'(x)\bigg(1+\frac{x\rho(x)}{\rho^{2}(x)}\bigg) < 0,  \ x < \bar x_2, \ x\not=0.
$$
We know also $A_-'(0,c,\tau)=0.5-\tau <0$.
Now, integrating by parts, we obtain 
$$
A_-(x,c,\tau)= \rho(x)+\frac{x\rho(x)+\int_{\rho(x)}^{0}vd\theta(v)}{\rho(x)}= \rho(x)+\frac{1}{\rho(x)}\int_{0}^{\rho(x)}\theta(v)dv=G(\rho(x)),$$
where $\theta(v):=\rho^{-1}(v)$ and $ G(z)=z+\int_{0}^{1}\theta(vz)dv$. 

Then, by Proposition \ref{nsw} and the formula for the Schwarzian derivative of the composition of two functions, we obtain
 $$(SA_-)(x,c,\tau)=(SG)(\rho(x))(\rho'(x))^{2}+(S\rho)(x)<(SG)(\rho(x))(\rho'(x))^{2}.$$
Thus  the negativity of  $SA_-$ will follow from the inequality   $(SG)(\rho(x))<0$. Since $A_-'(x,c,\tau) <0$ if and only if $G'(\rho(x))>0$, 
it suffices to show that $(SG)(\rho(x))<0$ when
$G'(\rho(x))>0$.  Now, in view of  Proposition \ref{nsw}, 
$$
\theta'''(\rho(x))=\frac{3(\rho''(x))^2-\rho'''(x)\rho'(x)}{(\rho'(x))^5} = \frac{-(S\rho)(x)}{(\rho'(x))^3} + \frac 32\frac{(\rho''(x))^2}{(\rho'(x))^5} <0. 
$$
Hence, 
$G'''(z)=\int_{0}^{1}v^{3}\theta'''(vz)dv <0, \ z = \rho(x),$
and therefore 
$
(SG)(\rho(x))< $ $G'''(\rho(x))/G'(\rho(x)) <0.
$
This completes the proof of Lemma \ref{31}. 
\end{proof}
Next, for $c\geq 2, \ \tau \in (1,3/2]$, we will also consider  the functions 
$$
R(x,c,\tau)= \frac{A'_+(0,c,\tau)x}{1-0.5A''_+(0,c,\tau)x/A'_+(0,c,\tau)}, $$
$$
 D(x,c,\tau)= \left\{ \begin{array}{ll} A_-(x,c,\tau) & \textrm{if $x\leq 0$}, \\ A_+(x,c,\tau) & \textrm{if $x \in [0,x_2]$},
\\ B(x,c,\tau) & \textrm{if $x\geq x_2$}. \end{array} \right.
$$
As the above discussion shows, $D(x,c,\tau)$ is strictly decreasing in $x \in \R$.  
From now on, we fix $\tau=3/2$ and set $A_\pm(x,c):= A_\pm(x,c,3/2),$  $B(x,c):= B(x,c,3/2), $ \ $D(x,c):= D(x,c,3/2), \ R(x,c):= R(x,c,3/2)$.  The strictly decreasing function 
$D(x,c)$ has the following additional nice property: 
\begin{proposition} \label{old} If $c \geq 2$ then $D(x,c) > R(x,c)$ for all $x >0$. 
\end{proposition}
\begin{proof} The above inequality follows from \cite[Corollary 2.7]{lpprt} if we take there ${\frak f}'(0)= r'(0)= - \tau =-3/2$. It should be observed that the definitions of functions $A, B, r, R$ in \cite{lpprt}  are identical to the definitions of  $A_+, B, r, R$ in this paper.   The only formal difference with \cite{lpprt} is the presence of  parameter $c$ in the expressions for the second derivatives of $A_+, r, R$ at $0$.  However, once these derivatives are positive, the proofs in \cite{lpprt} do not matter on their exact values , e.g. see Lemma 2.6 from \cite{lpprt}. 
\end{proof}
\begin{corollary} \label{fi} $F(x):= A_-(R(x)) <x$ for all $x >0$.  
\end{corollary} 
\begin{proof} 
By Lemma \ref{31},  $(S F)(x)= (SA_-)(R(x))(R'(x))^{2} <0$ for all $x$ from some open neighbourhood of $[0,+\infty)$.  Also  $SR\equiv 0$,  so that 
 $$0= (S R)(0)= -R'''(0) - 1.5 (R''(0))^2 = -R'''(0) - 1.5(A_-''(0))^2 .$$
Next, we  have that  
$F'(0) = -R'(0) =1,  \ F''(0)= A_-''(0)(R'(0))^2+ A_-'(0)R''(0)=0,$
$$F'''(0)= A_-'''(0)(R'(0))^3 + 3A_-''(0)R'(0)R''(0) + A_-'(0)R'''(0) =  (SA_-)(0) <0. 
$$
Therefore $F(x) <x$ for all small positive $x$.  Now, suppose that $F(z)=z$ for some leftmost positive $z$. Then $F'(z) \geq 1$ and therefore function $y=F'(x)>0,$ $ x \in[0,z],$ $ F'(0)=1,$ has a positive local minimum at some point $p \in (0,d)$. But then $F''(p)=0, \ F'''(p) \geq 0$, and in this way 
$(SF)(p) \geq 0$, a contradiction.   
\end{proof}

\section{Bounding relations and  the convergence of  semi-wavefronts} \label{BF}
As we have mentioned  in the introduction,  
 for each $c\geq 2$, equation  (\ref{17}) 
has at least one positive wave $u = \phi(\nu \cdot  x +ct),$ $|\nu| =1$, satisfying the boundary conditions $\phi(-\infty)=0$, $0 < \lim\inf_{t\to+\infty} \phi(t) \leq \lim\sup_{t\to+\infty} \phi(t) < +\infty$.  Clearly, $\phi$ satisfies
\begin{eqnarray} \label{twe2a}
\phi''(t) - c\phi'(t) + \phi(t)(1-
\phi(t-h)) =0,  \ h := c\tau, \ t \in \R. \end{eqnarray}
The change of variables $\phi(t) = e^{-x(t)}$ transforms  the latter equation  into  
\begin{equation} \label{et}
x''(t) - cx'(t) - (x'(t))^2 +(e^{-x(t-h)}-1)=0, \ t \in \R. 
\end{equation}
By Theorem 4 from \cite{HT},  $x(t)$ is sine-like oscillating 
around $0$.  More precisely, there exists an increasing  sequence $Q_j, \ j \geq 0,$ of zeros of $x(t)$ such that $x(t)<0$ on 
$(Q_0,Q_1)\cup(Q_2,Q_3) \cup \dots$ and $x(t)>0$ on 
$(-\infty, Q_0)\cup (Q_1,Q_2)\cup(Q_3,Q_4) \cup \dots$ Furthermore, $x(t)$ has exactly one critical point (hence, local extremum point) $T_j$ on each  interval $[Q_j,Q_{j+1}]$ and $T_j-Q_j <h$ for all $j$.  Hence, $y(t):=x'(t)$ does not change its sign on the intervals $(T_{j},T_{j+1}), \ j=0,1,2\dots$ and $y(T_j)=0$.  Therefore $y$ solves the boundary value problem 
\begin{equation} \label{rfc}
y'=y^2 +cy- g(t),\quad  y(T_j)=y(T_{j+1})=0,  
\end{equation}
where $c \geq 2$ and $g(t):= w(x(t-h))$ is $C^2$-smooth on $\R$.  
\begin{lemma}\label{L20}  For each integer $j\geq 0$, solution  
$y(t)$ has a unique critical point (absolute minimum point) $p_j \in [T_{2j+1},T_{2j+2}],$ and, 
for all $t \in (p_j,T_{2j+2})$, it holds  
$
y(t) > \rho(x(t-h))/h. 
$
Furthermore, for each non-increasing function $M=M(t),$ $t \in [Q_{2j},Q_{2j+2}],$ such that 
$x(t) \leq M(t),\  t \in [Q_{2j},Q_{2j+2}]$, it holds 
\begin{equation}\label{iM}
y(t) > \rho(M(t-h))/h,\  t \in  (T_{2j+1},T_{2j+2})\setminus\{p_j\}.
\end{equation}
Similarly,  for $j\geq 1$, solution  
$y(t)$ has a unique critical point (absolute maximum point) $q_j \in [T_{2j},T_{2j+1}],$ and, 
for all $t \in (q_j,T_{2j+1})$, it holds  
$
y(t) < \rho(x(t-h))/h. 
$
Furthermore, for each non-decreasing function $m=m(t), t \in [Q_{2j-1},Q_{2j+1}]$, such that 
$x(t) \geq m(t),\  t \in [Q_{2j-1},Q_{2j+1}]$, it holds 
$$
y(t) <   \rho(m(t-h))/h,\  t \in  (T_{2j},T_{2j+1})\setminus\{q_j\}. 
$$
\end{lemma}
\begin{proof} We will prove only the first  assertion of the lemma, the proof of the second statement being completely analogous.   So let us
consider the slope field for differential equation (\ref{rfc}).  Two zero isoclines
$$
\lambda_1(t) = \frac{-c- \sqrt{c^2+4g(t)}}{2} < -\frac c2 < \frac{-c+ \sqrt{c^2+4g(t)}}{2}:=\lambda_2(t)
$$
partition the plane $\R^2$ into three horizontal bands 
$$\Pi_1= \{(t,y): y\leq \lambda_1(t)\},  \Pi_2= \{(t,y): \lambda_1(t)\leq  y\leq \lambda_2(t)\},  \Pi_3= \{(t,y): y\geq  \lambda_2(t)\},  $$
limited by the graphs of  functions $y= \lambda_1(t), y= \lambda_2(t)$.  We observe that 
the portions of integral curves of (\ref{rfc}) belonging to the interior of domains $\Pi_1, \Pi_3$ [respectively, $\Pi_2$] are 
increasing [respectively, decreasing].  Since $y(T_{2j+2})=0$ and $g(T_{2j+2}) = \exp(-x(T_{2j+2}-h))- 1  < 0$ we find that 
$(T_{2j+2},0) \in {\rm Int}\, \Pi_3$, where Int\,$X$ denotes the interior part of the set $X$.   Similarly, $(T_{2j+1},0) \in {\rm Int}\, \Pi_2$ while the points $T_{2j+1}$ and $T_{2j+2}$ are separated by a unique zero $Q_{2j+1}+h$ of $y = \lambda_2(t)$ on $[T_{2j+1}, T_{2j+2}]$. As a consequence, 
the integral curve of each function $y(t)$ solving  (\ref{rfc}) never enters $\Pi_1$ and belongs to $\Pi_2\cup\Pi_3$.  Moreover, it is clear that $y'(t)>0$  on some maximal interval $(p_j, T_{2j+2})$ where $y(p_j) = \lambda_2(p_j), \ y'(p_j)=0$. Since clearly $0\leq y''(p_j) = -g'(p_j)$, the point 
$(p_j, \lambda_2(p_j))$  lies on the decreasing part of the graph $\Gamma$ of $y = \lambda_2(t)$ (observe that  $\lambda_2'(t)= g'(t)/ \sqrt{c^2+4g(t)})$. We claim that $(t,y(t))$ does not  cross 
$\Gamma$ again for all $t \in [T_{2j+1},p_j)$.  Indeed, otherwise there exists some $d \in  [Q_{2j+1}+h,p_j)$ such that $y(d)= \lambda_2(d)$ and therefore $y'(d)=0$ while $\lambda'(d_2) =  g'(d)/ \sqrt{c^2+4g(d)} <0$ since $g'(t) = -x'(t-h)\exp(-x(t-h)) <0,$  $t  \in  [Q_{2j+1}+h,p_j), \ g(p_j) \leq 0$.  This means that at the moment $t=d$ the integral curve of the solution $y=y(t)$ intersects  transversally $\Gamma$, enters  the domain $\Pi_3$ and is strictly increasing on $(d,p_j)$. Since $y=\lambda_2(t)$ is strictly decreasing on the same interval, we get a contradiction: $y(p_j) > \lambda_2(p_j)$.  

Hence, we have  the following description of the behaviour 
of each solution $y(t)$ to (\ref{rfc}) on  $[T_{2j+1},T_{2j+2}]$: there exists a point $p_j \in (T_{2j+1},T_{2j+2})$ such that 
\begin{enumerate}
\item[i)] $y'(t) >0, \ y(t) > \lambda_2(t)=  \rho(x(t-h))/h, \quad t \in (p_j,  T_{2j+2}]$; 
\item[ii)]  $y'(p_j) =0, \ y(p_j) = \lambda_2(p_j)$; 
\item[iii)] $y'(t) <0, \ y(t) < \lambda_2(t), \quad t \in [T_{2j+1}, p_j)$. 
\end{enumerate}
Finally, in order to justify  (\ref{iM}), we observe that $[Q_{2j},Q_{2j+2}] \supset [T_{2j+1}-h,T_{2j+2}-h]$. Therefore, since $\rho$ decreases on $\R$, we obtain that $\rho(x(t-h)) \geq \rho(M(t-h))$ for $t \in  [T_{2j+1},T_{2j+2}]$. Thus
the property i) implies (\ref{iM}) for all $t \in (p_j,  T_{2j+2}]$.  In particular, 
$
y(p_j) \geq \rho(M(p_j-h))/h.
$
Since, in addition,  $y(t)$ is strictly decreasing on 
$ [T_{2j+1},p_j)$,  $\rho(M(t-h))$ is non-decreasing on the same interval,  we conclude that (\ref{iM}) also holds for all $t \in [T_{2j+1},p_j)$. 
This completes  the proof of Lemma \ref{L20}.
\end{proof}
\begin{remark} For the oscillating semi-wavefront solutions  $x = x(t)$ of  equation (\ref{et}), the above result improves considerably the estimations of Lemma 20 from \cite{HT}. In order to obtain such an  improvement, here we have used  our knowledge of  slowly oscillating behaviour  of $g(t)$: this information was not relevant for the proof of Lemma 20. 
\end{remark}
\begin{corollary}
The profiles of oscillating semi-wavefronts to equation (\ref{17}) have a unique inflection point between each 
two consecutive extremum points. 
\end{corollary}
In the next stage of our studies,  we will  evaluate the extremal values $V_j=x(T_j)$ for $j\geq 1$ 
(it follows from \cite[Corollary 16]{HT} that  $V_0 \geq -ch$). 
\begin{lemma} \label{ogran} Let $c\geq 2, \ \tau \in (1,3/2]$ and $x(t) = -\ln \phi(t)$ oscillates on $[Q_0, +\infty)$.  Then \quad
$
V_{2j+1} \leq A_-(V_{2j},c,\tau), \ j \geq 0, $\
$V_{2j} \geq   B(V_{2j-1},c,\tau), \ j \geq 1. $  If, in addition,  $  V_{2j-1} \leq x_2,$ then  $V_{2j} \geq  A_+(V_{2j-1},c,\tau), \ j \geq 1. $
\end{lemma}
\begin{proof} 
As we know,  $V_1= x(T_1) >0$ with $T_1-Q_0 > h$ and $x'(T_1) =0, \  x(Q_1) =0,$ $T_1-Q_1 <h$.   Set $Q_{-1}=T_{-1}=-\infty$, it is clear that 
$x(s) \geq V_0$ for all $s \in [Q_{-1},Q_1]$.  On the other hand, due to Lemma \ref{L20}, 
we know that 
$$
x'(t) \leq \max_{s\in [T_0,T_1]}x'(s) \leq  \frac 1h \rho (\min_{s \in [Q_{-1},Q_1]}x(s)) \leq \frac 1h \rho(V_0), \quad t \in [T_{0},T_1],  
$$
and therefore 
$$
x(t) = -\int_t^{Q_1}x'(s)ds \geq - \frac 1h \int_t^{Q_1}\rho(V_0)ds = \frac{\rho(V_0)}{h}(t-Q_1)=\tilde m(t), \quad t \in [T_0, Q_1].
$$
In particular, $x(T_0)=V_0 \geq \tilde m(T_0)$ and therefore equation $\tilde m(t)=V_0$ has a root $t_1 \in [T_0,Q_1]$.  Since $V_0 <0$, we know from the first lines of the proof of Lemma \ref{31} that $t_1-Q_1= hV_0/\rho(V_0) > -h$.  Consider now the non-decreasing function
$$
m(t)= \left\{ \begin{array}{ll} \tilde m(t) & \textrm{if $t \in [t_1,Q_1]\subset (Q_1-h,Q_1]$}, \\ V_0 & \textrm{if $t \leq t_1$},
\end{array} \right.
$$
it is clear that $x(t) \geq m(t)$ for all $t \in [Q_{-1},Q_1].$
Therefore, by Lemma \ref{L20}, 
$$V_1 =  \int_{Q_1}^{T_1}x'(s)ds\leq \frac 1h\int_{Q_1}^{T_1}\rho(m(s-h))ds \leq \frac 1h\int_{Q_1-h}^{Q_1}\rho(m(s))ds=  A_-(V_0,c,\tau). 
$$
Next, consider $V_2= x(T_2) <0$, we have  $x'(t) <0$ on $(T_1,T_2)$,  $x'(T_2) =0,$ $ x(Q_2) =0$ and $T_2- Q_2<h$.  By Lemma \ref{L20} and Corollary \ref{co22},  
$$
x'(t)\geq   \min_{s\in [T_1,T_2]}x'(s) \geq  \frac 1h \rho(\max_{s \in [Q_0,Q_2]}x(s)) \geq  \frac{\rho(V_1)}{h} > \frac{r(V_1)}{h},  \quad  t\in [T_1,T_2], 
$$
and therefore 
$$
x(t) = -\int_t^{Q_2}x'(s)ds < \frac{r(V_1)}{h} (t-Q_2) = \tilde M(t), \quad t \in [T_1, Q_2).
$$
Since $\tilde M(t)$ is decreasing on $[Q_0,Q_2]$, it holds that  $\tilde M(T_1) > x(T_1)=V_1$ and $\max_{t \in [Q_0,Q_2]} x(t) = x(T_1)$, we have that $x(t) < \tilde M(t)$ for $t \in (Q_0,Q_2]$.  Then 
Lemma \ref{L20} yields
$$
V_2 =  \int_{Q_2}^{T_2}x'(s)ds \geq \frac 1h \int_{Q_2-h}^{Q_2}\rho(\tilde M(s))ds  > \frac 1h \int_{Q_2-h}^{Q_2}r(\tilde M(s))ds=  B(V_1,c). 
$$

Suppose now that $V_1 \leq x_2$. Let  $t=t_2$ solve  the equation
$
hV_1= r(V_1)(t-Q_2), 
$
then $t_2-Q_2= h V_1/r(V_1)  \geq -h$ (see the comments following the definition of $x_2$).  Consider  the non-increasing function
$$
M(t)= \left\{ \begin{array}{ll} \tilde M(t) & \textrm{if $t \in [t_2,Q_2]\subset [Q_2-h,Q_2]$}, \\ V_1 & \textrm{if $t \leq t_2$},
\end{array} \right.
$$
it is clear that $x(t) \leq M(t)$ for all $t \in [Q_{0},Q_2].$
By applying Lemma  \ref{L20} and Corollary \ref{co22}, we obtain
$$
V_2 =  \int_{Q_2}^{T_2}x'(s)ds\ > 
 \frac 1h \int_{Q_2-h}^{Q_2} \rho(M(s))ds  > 
 \frac 1h \int_{Q_2-h}^{Q_2} r(M(s))ds =
   A_+(V_1,c). 
$$
Finally, we can repeat the above arguments  to obtain 
similar estimations  for all $j>2$. 
This completes the proof of Lemma \ref{ogran}. \hfill \end{proof}

We are now in a position to finalise the proof of Theorem \ref{eWT}. 
Consider the following   finite limits 
$$0 \geq m_*= \liminf_{j \to +\infty} V_j = \liminf_{t \to +\infty}x(t), \quad 0 \leq  M_* = \limsup_{j \to +\infty} V_j=
\limsup_{t \to +\infty}x(t).$$ 
From  Lemmas \ref{31} and \ref{ogran}  we deduce that 
$
M_* \leq A_-(m_*, c, \tau)$ and   $m_* \geq  D(M_*,c,\tau)
$. 
Clearly, Theorem  \ref{eWT} will be proved if we show that $\tau \leq 1.5$ yields $M_*=0$. 
Since this implication was already proved for  $\tau \leq 1$ in \cite[Theorem 8]{HT}, we may assume that $\tau >1$. 
So let us suppose that $M_*>0$,  $\tau \in (1,3/2]$.  But then, due to inequalities (\ref{AAB}) and  Proposition \ref{old}, 
$
M_* \leq A_-(m_*, c, \tau) < A_-(m_*, c)$ and   $m_* \geq  D(M_*,c,\tau) > D(M_*,c)  > R(M_*)
$. Therefore we have $M_* < A_-(R(M_*))$.  However,  by Corollary \ref{fi}, $A_-(R(M_*)) <M_*$, a contradiction. 
 Hence, $M_*=0$ and the proof of Theorem  \ref{eWT} is completed.

\bibliographystyle{amsplain}

\end{document}